\documentclass[12pt]{article}
\usepackage{amsmath}
\usepackage{amssymb,latexsym}
\usepackage[only,fivrm]{rawfonts}
\usepackage{bbm}
\usepackage{enumerate}
\usepackage{amsthm}
\usepackage{multirow}
\usepackage{nccmath}
\newcommand{\BR}{\Lambda_n}
\newcommand{\hb}{\hat{\beta}}

\renewcommand{\O}{\ensuremath{{\cal O}}}

\newtheorem{theorem}{THEOREM}
\newtheorem{lemma}{LEMMA}

\begin{document}

\noindent AN ASYMPTOTIC LINEAR REPRESENTATION FOR THE BRESLOW ESTIMATOR
\vskip 3mm

\vskip 5mm
\noindent Hendrik P.~Lopuha\"{a} and Gabriela F.~Nane

\noindent Department of Applied Mathematics

\noindent Delft University of Technology

\noindent Mekelweg 4, 2628 CD, Delft, The Netherlands

\noindent G.F.Nane@tudelft.nl

\vskip 3mm

\vskip 3mm

\noindent ABSTRACT

We provide an asymptotic linear representation for the Breslow estimator
of the baseline cumulative hazard function in the Cox model.
Our representation consists of an average of independent random variables and
a term involving the difference between the maximum partial likelihood estimator and the
underlying regression parameter.
The order of the remainder term is arbitrarily close to $n^{-1}$.
\vskip 4mm

\noindent 1.   INTRODUCTION

The proportional hazards model is one of the most popular approaches to model right-censored time to event data in the presence of covariates. Cox (1972) introduced this semiparametric model and focused on estimating the underlying regression coefficients of the covariates. His estimator was later shown (Cox, 1975) to be a maximum partial likelihood estimator and its asymptotic properties were broadly studied (Tsiatis, 1981; Andersen et al., 1993; Oakes, 1977; Slud, 1982). Different functionals of the lifetime distribution are commonly investigated and the (cumulative) hazard function is of particular interest. In the discussion following the Cox's (1972) paper, Breslow proposed a nonparametric maximum likelihood estimator for the baseline cumulative hazard function.
Asymptotic properties of the Breslow estimator, such as consistency and the asymptotic distribution,
were derived by Tsiatis (1981) and Andersen et al. (1993). For an overview of the Breslow estimator, see Lin (2007).

 Estimators in unconditional censorship models such as the Kaplan--Meier and Nelson--Aalen estimators have received considerable attention, especially in the $1980$s.
Established large sample properties include consistency and asymptotic normality (Breslow and Crowley, 1974), rate of strong uniform consistency (Cs\"{o}rg\H{o} and Horv\'{a}th, 1983), strong approximation or Hungarian embedding (Burke et al., 1981), and linearization results (Lo and Singh, 1985).
Lo and Singh (1985) expressed the difference between the Kaplan--Meier estimator and the underlying distribution function
in terms of a sum of independent identically distributed random variables, almost surely,
with a remainder term of the order $n^{-3/4}(\log n)^{3/4}$, with $n$ denoting the sample size; this rate was later improved to $n^{-1}\log n$ by Lo et al. (1989).
To our knowledge, a strong approximation result for the Breslow estimator is unavailable in the literature.
Kosorok (2008) establishes a representation of the Breslow estimator in terms of counting processes.
Although this can be turned into an asymptotic linear representation similar to the one in Lo and Singh (1985),
the covariates are assumed to be in a bounded set and
the remainder term is only shown to be of the order~$o_p(n^{-1/2})$.

In this paper, we derive a similar linearization result for the Breslow estimator, i.e.,
we prove that the difference between the estimator $\Lambda_n$
and the cumulative baseline hazard function $\Lambda_0$
can be represented as a sum of independent random variables and a term involving the difference between the
regression parameter and its maximum partial likelihood estimator.
However, we allow unbounded covariates and we show that the remainder term is of the order $n^{-1}a_n^{-1}$,
where $a_n$ may be any sequence tending to zero.
As~$a_n$ can be chosen to converge to zero arbitrarily slowly, this means that the order of the remainder term is arbitrarily
close to $n^{-1}$.
The proof is based on empirical process theory, which allows the extension of our result
to related semi-parametric models, such as marginal regression models.
Our main motivation is isotonic estimation of the baseline distribution in the Cox model.
An example is the Grenander type estimator $\tilde{\lambda}_n$ for an increasing baseline hazard~$\lambda_0$,
considered in Lopuha\"{a} and Nane (2013),
which is defined as the left-hand slope of the greatest convex minorant of the Breslow estimator.
The limit behavior of $\tilde{\lambda}_n$ at a fixed point~$t_0$ essentially follows from the
limit behavior of the process
\[
t\mapsto n^{2/3}\left\{
(\Lambda_n-\Lambda_0)\left(t_0+n^{-1/3}t\right)
-
(\Lambda_n-\Lambda_0)\left(t_0\right)
\right\}.
\]
In the absence of a strong approximation result for the process $\Lambda_n-\Lambda_0$,
an alternative to obtain the limit process is to apply the results in Kim and Pollard (1990)
to the linear representation for~$\Lambda_n-\Lambda_0$, provided that the remaining terms in the
representation are of order smaller than $n^{-2/3}$.
This cannot be ensured by the representation in Kosorok (2008),
whereas the order~$n^{-1}a_n^{-1}$ can be chosen sufficiently small,
for suitable choices of~$a_n$.
Another application of our linear representation is that,
together with a linear representation for the maximum partial likelihood estimator,
a central limit theorem can be established for $\Lambda_n-\Lambda_0$.
Moreover, such a representation may also provide a means to estimate the variance of the Breslow estimator,
by using plug-in estimators.
A linear representation for the partial maximum likelihood estimator can be deduced
from results in Tsiatis (1981) or Kosorok (2008).

The paper is organized as follows.
The Cox model and the Breslow estimator are introduced in Section 2.
Section 3 is devoted to the main result of the paper and its proof as well as to preparatory lemmas.

\vskip 3mm

\noindent 2. BACKGROUND, NOTATION, AND ASSUMPTIONS

Let $X$ denote a positive random variable representing the survival time of a population of interest. The random variable $C$ denotes the censoring time. Now, define $T=\min(X,C)$ as the generic follow-up time and $\Delta=\{X\leq C\}$ as its corresponding indicator, where $\{\cdot\}$ denotes the indicator function. Suppose that at the beginning of the study, extra information such as sex, age, status of a disease, etc. is recorded for each subject as covariates. Let $Z$ denote a $p$-dimensional covariate vector. Therefore, suppose we observe the following independent, identically distributed triplets
$\left(T_i,\Delta_i,Z_i\right)$, with $i=1,\ldots,n$. The censoring mechanism is assumed to be non-informative. Moreover, given the covariate~$Z$, the survival time~$X$ is assumed to be independent of the censoring time~$C$. The $p$-dimensional covariate vector $Z$ is assumed to be time invariant and non-degenerate.

%The hazard function is of particular interest in survival analysis, as it represents an important feature of the survival time. In respect to this,
In the Cox model, the distribution of the survival time is related to the corresponding covariate by
\[
\lambda\left(x\mid z\right)=\lambda_0(x)\,\text{e}^{\beta_0' z},\quad x\in \mathbb{R}^{+},
\]
where $\lambda\left(x\mid z\right)$ is the hazard function for a subject with covariate vector $z\in\mathbbm{R}^p$, $\lambda_0$ represents the underlying baseline hazard function,
%corresponding to a subject with $z=0$
and $\beta_0\in \mathbbm{R}^p$ is the vector of the underlying regression coefficients. Conditionally on $Z=z$, the survival time~$X$ is assumed to be a nonnegative random variable, with an absolutely continuous distribution function $F(x \mid z)$ with density $f(x\mid z)$. The same assumptions hold for the censoring variable~$C$ and its distribution function $G$. Let~$H$ be the distribution function of the follow-up time~$T$ and let~$\tau_H=\inf\{t:H(t)=1\}$ be the end point of the support of~$H$. Moreover, let~$\tau_F$ and~$\tau_G$ be the end points of the support of~$F$ and~$G$, respectively.
We employ the usual assumptions for deriving large sample properties of Cox proportional hazards estimators (Tsiatis, $1981$):
\begin{description}
\item (A1)
$\tau_{H}=\tau_G<\tau_F$.
\item (A2)
There exists $\varepsilon>0$ such that
\[
\sup_{|\beta-\beta_0 |\leq \varepsilon} \mathbbm{E}\left[ |Z|^2\, \text{e}^{2\beta' Z}\right]<\infty,
\]
where $|\cdot|$ denotes the Euclidean norm.
\end{description}
Let $X_{(1)}<\cdots<X_{(m)}$ denote the ordered, observed survival times.
Cox (1972, 1975) introduced the proportional hazards model and proposed the partial likelihood estimator $\hb$ as an estimator for the underlying regression coefficients $\beta_0$. Breslow (Cox, 1972) focused on estimating the baseline cumulative hazard function, $\Lambda_0(x)=\int_0^x \lambda_0(u)\,\mathrm{d}u$, and proposed

\begin{equation}
\label{breslow_trad}
\BR(x)=
\sum_{i\mid X_{(i)}\leq x}
\frac{d_i}{\sum_{j=1}^n \{T_j\geq X_{(i)}\}\,\text{e}^{\hb' Z_j}},
\end{equation}
as an estimator for $\Lambda_0$,
where $d_i$ is the number of events at $X_{(i)}$ and $\hb$ is the partial maximum likelihood estimator of the regression coefficients. The estimator $\BR$ is most commonly referred to as the Breslow estimator. Under the assumption of a piecewise constant baseline hazard function and assuming that all the censoring times are shifted to the preceding observed survival time, Breslow showed that the partial maximum likelihood estimator $\hb$ along with the baseline cumulative hazard estimator $\BR$ can be obtained by jointly maximizing the full loglikelihood function.

Let
\begin{equation}
\label{Phi}
\begin{split}
\Phi(\beta,x)
&=
\int \{u\geq x\}\, \text{e}^{\beta' z}\, \mathrm{d}P(u,\delta,z),\\
\Phi_n(\beta,x)
&=
\int \{u\geq x\}\, \text{e}^{\beta' z}\, \mathrm{d}P_n(u,\delta,z),
\end{split}
\end{equation}
where $P$ is the underlying probability measure corresponding to the distribution of $(T,\Delta,Z)$ and $P_n$ is the empirical measure of the triplets $(T_i,\Delta_i,Z_i)$, for $i=1,2,\ldots,n$.
Furthermore, let $H^{uc}(x)=\mathbbm{P}(T\leq x, \Delta=1)$ be the sub-distribution function of the uncensored observations.
Then, using the derivations in Tsiatis (1981), it can be deduced that
\begin{equation}\label{base_haz}
\lambda_0(u)=\frac{ \mathrm{d}H^{uc}(u)/\mathrm{d}u }{\Phi(\beta_0,u)}.
\end{equation}
Consequently, it can be derived that
\begin{equation}
\label{cum_base_haz}
\Lambda_0(x)=\int \frac{\delta\{ u \leq x\}}{\Phi(\beta_0,u)}\,\mathrm{d}P(u,\delta,z).
\end{equation}
From (A1) it follows that $\Lambda_0(\tau_H)<\infty$.
An intuitive baseline cumulative hazard function estimator is obtained
by replacing $\Phi$ in \eqref{cum_base_haz} by $\Phi_n$ and by plugging in $\hb$,
which yields exactly the Breslow estimator in \eqref{breslow_trad},
\begin{equation}
\label{Breslow}
\Lambda_n(x)=\int \frac{\delta\{u\leq x\}}{\Phi_n(\hb,u)}\ \mathrm{d}P_n(u,\delta,z).
\end{equation}
Kosorok (2008) established strong uniform consistency for the Breslow estimator and the process convergence of $\sqrt{n}(\BR-\Lambda_0)$,
yet under the strong assumption of bounded covariates.
Using standard empirical processes methods, Lopuha\"{a} and Nane (2013) established
strong uniform consistency at rate $n^{-1/2}$ for the Breslow estimator under the relatively mild conditions (A1) and (A2).

\vskip 3mm

\noindent 3.  ASYMPTOTIC REPRESENTATION

The following two lemmas will be used in proving the main result of the paper.
\begin{lemma}
\label{lemma:rate Phin}
Suppose that condition (A2) holds and let $\Phi_n$ and $\Phi$ be defined in~\eqref{Phi}.
With $\varepsilon>0$ taken from (A2), for $|\beta-\beta_0|<\varepsilon$, let
\begin{equation}
\label{def:D}
\begin{split}
D^{(1)}(\beta,x)
&=
\frac{\partial\Phi(\beta,x)}{\partial\beta}
=
\int\{ u\geq x \}\, z \, \text{e}^{\beta' z}\, \mathrm{d} P(u,\delta,z)\,\in \mathbbm{R}^p,\\
D_n^{(1)}(\beta,x)
&=
\frac{\partial\Phi_n(\beta,x)}{\partial\beta}
=
\int\{ u\geq x \}\, z \, \text{e}^{\beta' z}\, \mathrm{d}P_n(u,\delta,z)\,\in \mathbbm{R}^p.
\end{split}
\end{equation}
Then,
\begin{equation}
\label{eq:rate Phin}
\begin{split}
\sqrt{n}\sup_{x\in\mathbbm{R}}\left| \Phi_n(\beta_0,x)-\Phi(\beta_0,x) \right|&=\O_p(1),\\
\sqrt{n}\sup_{x\in\mathbbm{R}}\left| D_n^{(1)}(\beta_0,x)-D^{(1)}(\beta_0,x) \right|&=\O_p(1).
\end{split}
\end{equation}
\end{lemma}
\begin{proof}
Consider the class of functions $\mathcal{G}=\{g(u,z;x): x\in\mathbb{R}\}$,
where, for each $x\in\mathbb{R}$ and $\beta_0\in\mathbbm{R}^p$ fixed,
\[
g(u,z;x)=\{u\geq x\}\exp(\beta_0'z)
\]
is a product of an indicator and a fixed function.
It follows that $\mathcal{G}$ is a Vapnik--\u{C}ervonenkis (VC)-subgraph class (Lemma 2.6.18 in van der Vaart and Wellner, 1996)
and its envelope $G=\exp(\beta_0'z)$ is square integrable under condition~(A2).
Standard results from empirical process theory (van der Vaart and Wellner, 1996) yield that the class of functions $\mathcal{G}$ is a Donsker class, i.e.,
\[
\sqrt{n}\int g(u,z;x)\,\mathrm{d}(P_n-P)(u,\delta,z)=\O_p(1),
\]
so that the first statement in~\eqref{eq:rate Phin} follows by the continuous mapping theorem.
To prove the second statement, it suffices to consider each $j$th coordinate, for $j=1,\ldots,p$, fixed.
In this case, we deal with the class~$\mathcal{G}_j=\{g_j(u,z;x): x\in\mathbb{R}\}$, where
\[
g_j(u,z;x)=\{u\geq x\}z_j\exp(\beta_0'z).
\]
From here the argument is exactly the same, which proves the lemma.
\end{proof}
\begin{lemma}
\label{lemma:jon}
Assume (A1) and (A2).
Then, for all $M\in(0,\tau_H)$,
\[
a_n n\sup_{x\in[0,M]}
\left|
\int
\delta
\{u\leq x\}
\left( \frac{1}{\Phi_n(\beta_0,u)}-\frac{1}{\Phi(\beta_0,u)} \right)\mathrm{d}(P_n-P)(u,\delta,z)\right|
=\O_p(1),
\]
for any sequence $a_n=o(1)$.
\end{lemma}

\begin{proof}
Consider the class of functions
$\mathcal{F}_n=\left\{ f_n(u,\delta,z;x): 0\leq x\leq M \right\}$,
where
\[
f_n(u,\delta,z;x)
=
\delta
\{u\leq x\}
\left( \frac{1}{\Phi_n(\beta_0,u)}-\frac{1}{\Phi(\beta_0,u)}\right).
\]
Correspondingly, consider the class $\mathcal{G}_{n,M,\alpha}$ consisting of functions
\[
g(u,\delta,z;y,\Psi)=
\delta\{u\leq y\}
\left( \frac{1}{\Psi(u)}-\frac{1}{\Phi(\beta_0,u)} \right),
\]
where $0\leq y \leq M$ and $\Psi$ is nonincreasing left continuous, such that
\[
\Psi(M)
\geq
K,
\quad
\sup_{u\in[0,M]}\left| \Psi(u)-\Phi(\beta_0,u) \right|\leq \alpha,
\]
where $K=\Phi(\beta_0,M)/2$.
Then, for any $\alpha>0$, we have $\mathbbm{P}(\mathcal{F}_n\subset \mathcal{G}_{n,M,\alpha}) \to 1$,
by Lemma~\ref{lemma:rate Phin}.
Furthermore, the class $\mathcal{G}_{n,M,\alpha}$ has envelope $G(u,\delta,z)=\alpha/K^2$.
Since the functions in $\mathcal{G}_{n,M,\alpha}$ are
 products of indicators and a difference of bounded monotone functions,
its entropy with bracketing satisfies
\[
\log N_{[\,]}(\varepsilon,\mathcal{G}_{n,M,\alpha},L_2(P))\lesssim \frac{1}{\varepsilon},
\]
see e.g., Theorem~2.7.5 in van der Vaart and Wellner (1996) and Lemma~9.25 in Kosorok (2008).
Hence, for any $\delta>0$, the bracketing integral
\[
J_{[\,]}(\delta,\mathcal{G}_{n,M,\alpha},L_2(P))=\int_0^\delta\sqrt{1+\log N_{[\,]}(\varepsilon\|G\|_2,\mathcal{G}_{n,M,\alpha},L_2(P))}\,\mathrm{d}\varepsilon
<\infty.
\]
By Theorem~2.14.2 in van der Vaart and Wellner (1996), we have
\[
\begin{split}
\mathbb{E}\left\|\sqrt{n}\int g(u,\delta,z;y,\Psi)\mathrm{d}(P_n-P)(u,\delta,z) \right\|_{\mathcal{G}_{n,M,\alpha}}
&\leq
J_{[\,]}(1,\mathcal{G}_{n,M,\alpha},L_2(P))\|G\|_{P,2}
=
\O(\alpha),
\end{split}
\]
where $\|\cdot\|_{\mathcal{F}}$ denotes the supremum over the class of functions $\mathcal{F}$.
Now, let $a_n=o(1)$.
Then, according to~\eqref{eq:rate Phin},
\[
a_n\sqrt{n}\sup_{x\in\mathbbm{R}}\left| \Phi_n(\beta_0,x)-\Phi(\beta_0,x) \right|=o_p(1).
\]
Therefore, if we choose $\alpha=n^{-1/2}a_n^{-1}$, this gives
\[
\mathbb{E}
\left\|
\int g(u,\delta,z;y,\Psi)\mathrm{d}(P_n-P)(u,\delta,z)
\right\|_{\mathcal{G}_{n,M,\alpha}}
=
\O((na_n)^{-1})
\]
and hence, by the Markov inequality, this proves the lemma.
\end{proof}

The asymptotic linear representation of the Breslow estimator is provided by the next theorem.
\begin{theorem}
\label{theorem:linearization}
Assume (A1) and (A2).
Let~$\Phi$ and~$D^{(1)}$ be defined in~\eqref{Phi} and~\eqref{def:D}.
Then, for all $M\in(0,\tau_H)$ and $x\in[0,M]$,
\[
\BR(x)-\Lambda_0(x)
=
\frac{1}{n}\sum_{i=1}^n \xi(T_i,\Delta_i,Z_i;x)+
(\hb-\beta_0)'A_0(x)
+R_n(x),
\]
where $\hb$ is the maximum partial likelihood estimator,
\begin{equation}
\label{def:A_0}
A_0(x)=
\int_0^x\frac{D^{(1)}(\beta_0,u)}{\Phi(\beta_0,u)}\lambda_0(u)\,\mathrm{d}u
\end{equation}
and
\[
\xi(t,\delta,z;x)
=
-\mathrm{e}^{\beta_0'z}\int_0^{x\wedge t} \frac{\lambda_0(u)}{\Phi(\beta_0,u)}\,\mathrm{d}u
+
\frac{\delta\{t\leq x\}}{\Phi(\beta_0,t)}
\]
and $R_n$ is such that
\[
\sup_{ x\in[0,M]}
\left|R_n(x)\right|=\O_p(n^{-1}a_n^{-1}),
\]
for any sequence $a_n=o(1)$.
\end{theorem}

\begin{proof}
For $\beta\in \mathbbm{R}^p$, define
\[
\BR(\beta,x)= \int \frac{\delta\{u\leq x\}}{\Phi_n(\beta,u)}\,\mathrm{d}P_n(u,\delta,z).
\]
Hence, the Breslow estimator in \eqref{Breslow} can also be written as $\BR(\hb,x)$.
For $x\in[0,M]$, consider the following decomposition
\[
\BR(x)-\Lambda_0(x)
=
T_{n1}(x)+T_{n2}(x),
\]
where $T_{n1}(x)=\BR(\hb,x)-\BR(\beta_0,x)$ and
$T_{n2}(x)=\BR(\beta_0,x)-\Lambda_0(x)$.

For the term $T_{n1}$, first notice that a Taylor expansion of $\BR(\cdot,x)$ around $\beta_0$ yields that
\begin{equation}
\label{eq:expansion}
\BR(\hb,x)-\BR(\beta_0,x)
=
-\big(\hb-\beta_0\big)'A_{n}(x)
+
\frac{1}{2}\big(\hb-\beta_0\big)' R_{n1}(x)\big(\hb-\beta_0\big),
\end{equation}
where the vector $A_n$ and matrix $R_{n1}$ are given by
\begin{align}
\label{a_n}
A_{n}(x)
&=
\int \delta\{u\leq x\}\frac{D^{(1)}_n(\beta_0,u)}{\Phi^2_n(\beta_0,u)}\ \mathrm{d}P_n(u,\delta,z),\\
R_{n1}(x)
&=
\int \delta\{u\leq x\}
\frac{2 D_n^{(1)}(\beta^*,u)D_n^{(1)}(\beta^*,u)'-D_n^{(2)}(\beta^*,u)\Phi_n(\beta^* ,u)}{\Phi^3_n(\beta^*,u)}\,
\mathrm{d}P_n(u,\delta,z),\nonumber
\end{align}
for some $|\beta^*-\beta_0|\leq |\hb-\beta_0|$,
with $D_n^{(1)}$ as defined in~\eqref{def:D} and
\[
D_n^{(2)}(\beta,x)
=
\frac{\partial^2\Phi_n(\beta,x)}{\partial\beta^2}
=
\int\{ u\geq x \}\, zz' \, \text{e}^{\beta' z}\, \mathrm{d}P_n(u,\delta,z)\,\in \mathbbm{R}^p\times\mathbbm{R}^p.
\]
We define $D^{(2)}(\beta,x)$ similarly, with $P_n$ replaced by $P$. \\
According to (A2), we have
$|D^{(1)}(\beta_0,x)|\leq \mathbb{E}\left[|Z|\exp(\beta_0'Z)\right]<\infty$,
for all $x\in \mathbb{R}$, and similarly
\[
|D_{n}^{(1)}(\beta_0,x)|
\leq
\frac1n\sum_{i=1}^n
|Z_i|\text{e}^{\beta_0'Z_i}
\to
\mathbb{E}
\left[
|Z|\text{e}^{\beta_0'Z}
\right]
<\infty,
\]
with probability one.
Likewise, $|D^{(2)}(\beta_0,x)|<\infty$ and
\[
|D_{n}^{(2)}(\beta_0,x)|
\leq
\frac1n\sum_{i=1}^n
|Z_i|^2\text{e}^{\beta_0'Z_i}
\to
\mathbb{E}
\left[
|Z|^2\text{e}^{\beta_0'Z}
\right]<\infty,
\]
with probability one.
Furthermore, for all $x\in[0,M]$,
\[
0<\Phi(\beta_0,M)\leq \Phi(\beta_0,x)\leq \Phi(\beta_0,0)=\mathbb{E}\left[\text{e}^{\beta_0'Z}\right]<\infty
\]
and $\Phi_n(\beta_0,M)\leq \Phi_n(\beta_0,x)\leq \Phi_n(\beta_0,0)$, where $\Phi_n(\beta_0,M)\to\Phi(\beta_0,M)$ and $\Phi_n(\beta_0,0)\to\Phi(\beta_0,0)$, with probability one.
It follows that there exist constants $K_1, K_2>0$, such that for all $x\in[0,M]$,
\begin{equation}
\label{eq:D1D2Phi bound}
|D^{(1)}(\beta_0,x)|\leq K_2,
\quad
|D^{(2)}(\beta_0,x)|\leq K_2,
\quad
K_1\leq \Phi(\beta_0,x)\leq K_2
\end{equation}
and for $n$ sufficiently large,
\begin{equation}
\label{eq:D1D2Phin bound}
|D_n^{(1)}(\beta_0,x)|\leq K_2,
\quad
|D_n^{(2)}(\beta_0,x)|\leq K_2,
\quad
K_1\leq \Phi_n(\beta_0,x)\leq K_2,
\end{equation}
with probability one.
According to~\eqref{base_haz},
\begin{equation}
\label{eq:relationPlambda}
\frac{\delta}{\Phi(\beta_0,u)}\ \mathrm{d}P(u,\delta,y)
=
\frac{\mathrm{d}H^{uc}(u)}{\Phi(\beta_0,u)}
=
\lambda_0(u)\,\mathrm{d}u,
\end{equation}
so that $A_0$, as defined in~\eqref{def:A_0}, is equal to
\[
A_0(x)
=
\int
\delta\{ u\leq x\}\frac{D^{(1)}(\beta_0,u)}{\Phi^2(\beta_0,u)}\,\mathrm{d}P(u,\delta,z)\in \mathbb{R}^{p},
\]
Then, for the $A_{n}$ term in~\eqref{eq:expansion}, it can be deduced that
\[
\begin{split}
\sup_{0\leq x\leq M}
|A_{n}(x)-A_0(x)|
&\leq
\sup_{0\leq u\leq M}
\left|
\frac{D_n^{(1)}(\beta_0,u)}{\Phi_n^2(\beta_0,u)}-\frac{D^{(1)}(\beta_0,u)}{\Phi^2(\beta_0,u)}
\right|
\\
&\qquad+
\sup_{0\leq x\leq M}
\left|
\int\delta\{u\leq x\}
\frac{D^{(1)}(\beta_0,u)}{\Phi^2(\beta_0,u)}
\mathrm{d}(P_n-P)(u,\delta,z)
\right|.
\end{split}
\]
By~\eqref{eq:D1D2Phi bound} and~\eqref{eq:D1D2Phin bound}, the first term on the right hand side is bounded by
\[
\frac{1}{K_1^2}
\sup_{0\leq x\leq M}
\left|D_n^{(1)}(\beta_0,x)-D^{(1)}(\beta_0,x)\right|
+
\frac{2K_2^2}{K_1^4}
\sup_{0\leq x\leq M}
\left|\Phi_n(\beta_0,x)-\Phi(\beta_0,x)\right|,
\]
which is of the order $\O_p(n^{-1/2})$, by Lemma~\ref{lemma:rate Phin}.
For the second term on the right hand side, for each $j=1,\hdots,p$, fixed,
consider the class $\mathcal{G}_j=\{g_j(u,\delta;x):x\in[0,M]\}$,
consisting of functions
\[
g_j(u,\delta;x)
=
\delta\{u\leq x\} \frac{D_j^{(1)}(\beta_0,u)}{\Phi^2(\beta_0,u)},
\]
where $D_j^{(1)}$ denotes the $j$th coordinate of $D^{(1)}$.
Now, each $g_j(u,\delta;x)$ is the product of indicators and a fixed uniformly bounded function.
Standard results from empirical process theory (van der Vaart and Wellner, 1996) give that the class $\mathcal{G}_j$ is Donsker.
As in the proof of Lemma~\ref{lemma:rate Phin}, we find that for every $j=1,\hdots,p$,
\[
\sqrt{n}
\sup_{0\leq x\leq M}
\left|
\int g_j(u,\delta;x)\,\mathrm{d}(P_n-P)(u,\delta,z)
\right|=\O_p(1).
\]
It follows that
\[
\sup_{0\leq x\leq M}
|A_{n}(x)-A_0(x)|
=
\O_p(n^{-1/2}).
\]
and we can conclude that
\[
\big(\hb-\beta_0\big)' A_{n}(x)
=
\big(\hb-\beta_0\big)'A_0(x)
+
R_{n2}(x),
\]
where $R_{n2}(x)=\O_p(n^{-1})$, uniformly for $x\in[0,M]$, since $\hb-\beta_0=\O_p(n^{-1/2})$ (Tsiatis, 1981).
For the term containing $R_{n1}$, first observe that,
according to~\eqref{eq:D1D2Phin bound},
for $n$ sufficiently large,
\[
\sup_{u\in[0,M]}\left|
\frac{2 D_n^{(1)}(\beta^*,u)D_n^{(1)}(\beta^*,u)'-D_n^{(2)}(\beta^*,u)\Phi_n(\beta^* ,u)}{\Phi^3_n(\beta^*,u)}
\right|=\O(1),
\]
almost surely, so that
\[
\sup_{0\leq x\leq M}
\left|\frac{1}{2}
\big(\hb-\beta_0\big)'
R_{n1}(x)
\big(\hb-\beta_0\big)\right|
=
\O_p(n^{-1}).
\]
Concluding,
\begin{equation}
\label{d_n,1}
T_{n1}(x)
=
\big(\hb-\beta_0\big)'A_0(x)
+
\O_p(n^{-1}),
\end{equation}
uniformly in $x\in[0,M]$.
Proceeding with $T_{n2}$, write
\[
T_{n2}(x)
=
\BR(\beta_0,x)-\Lambda_0(x)
=
B_n(x)+C_n(x)+R_{n3}(x)+R_{n4}(x),
\]
where
\[
\label{def:AB}
\begin{split}
B_n(x)
&=
\int \delta\{u\leq x\} \frac{\Phi(\beta_0,u)-\Phi_n(\beta_0,u)}{\Phi^2(\beta_0,u)}\, \mathrm{d}P(u,\delta,z), \\
C_n(x)
&=
\int \frac{\delta\{u\leq x\}}{\Phi(\beta_0,u)}\, \mathrm{d}(P_n-P)(u,\delta,z),\\
R_{n3}(x)
&=
\int \delta\{u\leq x\} \left(\frac{1}{\Phi_n(\beta_0,u)} -\frac{1}{\Phi(\beta_0,u)}\right)\,\mathrm{d}(P_n-P)(u,\delta,z),\\
R_{n4}(x)
&=
\int \delta\{u\leq x\}
\frac{[\Phi(\beta_0,u)-\Phi_n(\beta_0,u)]^2}{\Phi^2(\beta_0,u)\Phi_n(\beta_0,u)} \, \mathrm{d}P(u,\delta,z).
\end{split}
\]
For the dominating term in $T_{n2}$, we can write
\[
\begin{split}
B_n(x)+C_n(x)
&=
-\int \delta\{u\leq x\} \frac{\Phi_n(\beta_0,u)}{\Phi^2(\beta_0,u)}\ \mathrm{d}P(u,\delta,z) +
\int \frac{\delta\{u \leq x\}}{\Phi(\beta_0,u)}\ \mathrm{d}P_n(u,\delta,z)\\
&=
\frac{1}{n} \sum_{i=1}^n
\xi(T_i,\Delta_i,Z_i;x),
\end{split}
\]
where
\[
\xi(t,\delta,z;x)
=
-\int \gamma\{u \leq x\}
\frac{\{t\geq u\}\mathrm{e}^{\beta_0'z}}{\Phi^2(\beta_0,u)}\,
\mathrm{d}P(u,\gamma,y)+ \frac{\delta\{t\leq x\}}{\Phi(\beta_0,t)}.
\]
Using~\eqref{eq:relationPlambda}, we conclude that
\[
\xi(t,\delta,z;x)
=
-\mathrm{e}^{\beta_0'z}\int_0^{x\wedge t} \frac{\lambda_0(u)}{\Phi(\beta_0,u)}\,\mathrm{d}u
+
\frac{\delta\{t\leq x\}}{\Phi(\beta_0,t)}.
\]
For the remainder terms, it follows by Lemma~\ref{lemma:jon}, that
for any sequence $a_n=o(1)$,
\begin{equation}
\label{eq:bound Rn3}
\sup_{0\leq x\leq M}
\left|
R_{n3}(x)
\right|
=
\O_p(n^{-1}a_n^{-1}).
\end{equation}
To treat $R_{n4}$, note that
\[
\left| R_{n4}(x) \right|\leq \frac{1}{\Phi^2(\beta_0,M)}\frac{1}{\Phi_n(\beta_0,M)}\sup_{x\in\mathbbm{R}} |\Phi_n(\beta_0,x)-\Phi(\beta_0,x)|^2,
\]
so that by~\eqref{eq:rate Phin} and~\eqref{eq:D1D2Phin bound},
\[
\sup_{0\leq x\leq M}\left| R_{n4}(x)\right|=\O_p(n^{-1}).
\]
Together with~\eqref{d_n,1} and~\eqref{eq:bound Rn3}, this proves the theorem.
\end{proof}

In the special case of no covariates, i.e., $\beta_0=\hb=0$, it follows that
\[
\Phi(\beta_0,x)=1-H(x)
\]
and

\begin{align*}
\xi(t,\delta,z;x)
&=
-\mathrm{e}^{\beta_0'z}\int_0^{x\wedge t} \frac{\lambda_0(u)}{\Phi(\beta_0,u)}\,\mathrm{d}u
+
\frac{\delta\{t\leq x\}}{\Phi(\beta_0,t)}
=
-\int_0^{x\wedge t} \frac{\mathrm{d}H^{uc}(u)}{[1-H(u)]^2}+\frac{\delta\{t\leq x\}}{1-H(t)}.
\end{align*}

This means that Theorem~\ref{theorem:linearization} retrieves a result similar to Lemma 2.1 in Lo et al.~(1989).

The rate at which the error term $R_n$ tends to zero becomes faster as $a_n$ tends to zero more slowly.
If $a_n=1/\log n$, we obtain the same rate as the error term in Lemma 2.1 in Lo et al.~(1989).
However, they obtain the order $\O(n^{-1}\log n)$ almost surely, whereas Theorem~\ref{theorem:linearization},
with the choice $a_n=1/\log n$, only provides this order in probability.
Also, the sequence $a_n$ may be chosen to converge to zero arbitrarily slowly.
This means that the order $\O_p(n^{-1}a_n^{-1})$ of $R_n$ is arbitrarily close to $\O_p(n^{-1})$.

Using a linear representation for $\hat{\beta}-\beta_0$, a full linearization for the Breslow estimator can
be obtained.
Such a linear representation can be deduced from the proof of Theorem 3.2 in Tsiatis (1981) or
from an application of Theorem 2.11 in Kosorok (2008); see also Section 4.2.1 in Kosorok (2008).
As a consequence, Theorem 1 together with the expansion of $\hat{\beta}-\beta_0$
can be used to establish a central limit theorem for the Breslow estimator,
as well as to estimate the limiting covariance structure,
by using plug-in estimators.
For example, the term $A_0$ in the linear expression can be estimated consistently by $A_n$ in~\eqref{a_n}.
\vskip 3mm

\noindent ACKNOWLEDGEMENTS

We thank the two anonymous reviewers for their valuable comments and suggestions that helped to improve
the original version of the paper.
\vskip 3mm

\noindent REFERENCES

\noindent Andersen, P.~K., Borgan, O., Gill, R.~D., Keiding, N. (1993). {\it Statistical Models Based on

Counting Processes}. New York: Springer.

\noindent Breslow, N., Crowley, J. (1974). A large sample study of the life table and product limit

estimates under random censorship. {\it Ann. Statist.}  2:437--453.

\noindent Burke, M.~D., Cs{\"o}rg{\H{o}}, S., Horv{\'a}th, L. (1981). Strong approximations of some biometric

estimates under random censorship. {\it Z. Wahrsch. Verw. Gebiete} 56:87--112.

\noindent Cox, D.~R. (1972). Regression models and life-tabels (with discussion). {\it J. Roy. Statist. Soc. Ser. B.} 34:

187--220.

\noindent Cox, D.~R. (1975). Partial likelihood. {\it Biometrika} 62:269--276.

\noindent Cs{\"o}rg{\H{o}}, S., Horv{\'a}th, L. (1983). The rate of strong uniform consistency for the product-limit

estimator. {\it Z. Wahrsch. Verw. Gebiete} 62:411--426.

\noindent Kim, J., Pollard, D. (1990). Cube root asymptotics. {\it Ann. Statist.} 18:191--219.

\noindent Kosorok, M.~R. (2008). {\it Introduction to Empirical Processes and Semiparametric Inference}.

New York: Springer.

\noindent Lin, D.~Y. (2007). On the Breslow estimator. {\it Lifetime Data Anal.} 13:471--480.

\noindent Lo, S.~H., Mack, Y.~P., Wang, J.~L. (1989). Density and hazard rate estimation for censored

data via strong representation of the Kaplan--Meier Estimator. {\it Probab. Th. Rel. Fields}

80:461--473.

\noindent Lo, S.~H., Singh, K. (1985). The product-limit estimator and the bootstrap: Some asymptotic

representations. {\it Prob. Th. Rel. Fields} 71:455--465.

\noindent Lopuha\"{a}, H.~P., Nane, G.~F. (2013). Shape constrained nonparametric estimators of the

baseline distribution in Cox proportional hazards model. {\it To appear in Scand. J. Statist.}

\noindent Oakes, D. (1977). The asymptotic information in censored survival data. {\it Biometrika} 64:441--

448.

\noindent Slud, E.~V. (1982). Consistency and efficiency of inferences with the partial likelihood.

{\it Biometrika} 69:547--552.

\noindent Tsiatis, A. (1981). A large sample study of Cox's regression model. {\it Ann. Statist.} 9:93--108.

\noindent van der Vaart, A.~W., Wellner, J.~A. (1996). {\it Weak Convergence and Empirical Processes}.

New York: Springer.

\end{document}